\documentclass[11pt]{article}

\usepackage[english]{babel}
\usepackage{upref,amsfonts,amsxtra,latexsym,color}
\usepackage{amssymb,amsmath,amsthm,a4wide,epsf,mathrsfs,verbatim,hyperref}
\usepackage[all]{xy}
\usepackage[
       figurewithin=section,
       tablewithin=section
    ]{caption}
 
\newtheorem{innercustomthm}{Theorem}
\newenvironment{customthm}[1]
  {\renewcommand\theinnercustomthm{#1}\innercustomthm}
  {\endinnercustomthm}
  
 \newtheorem{innercustomlm}{Lemma}
\newenvironment{customlm}[1]
  {\renewcommand\theinnercustomlm{#1}\innercustomlm}
  {\endinnercustomlm}

\newtheorem{theorem}{Theorem}[subsection]

\newtheorem{corollary}[theorem]{Corollary}
\newtheorem{proposition}[theorem]{Proposition}

\theoremstyle{definition}

\newtheorem{remark}[theorem]{Remark}
\newtheorem{example}[theorem]{Example}

\numberwithin{equation}{section}
\numberwithin{theorem}{section}

\newcommand{\cA}{{\cal A}}

\newcommand{\cR}{{\cal R}}

\newcommand{\C}{{\mathbb C}}

\newcommand{\T}{{\mathbb T}}
\newcommand{\Q}{{\mathbb Q}}

\newcommand{\R}{{\mathbb R}}
\newcommand{\Cs}{{$C^*$-al\-ge\-bra}}

\DeclareMathOperator{\tr}{tr}

\date{\empty}
\title{Non-closure of quantum correlation matrices and factorizable channels that require infinite dimensional ancilla}

\author{Magdalena Musat$^*$ and Mikael R\o rdam\thanks{This research was supported by a travel grant from the Carlsberg Foundation, and by a research grant from the Danish Council for Independent Research, Natural Sciences. This work was carried out in Spring 2018, while the authors were visiting the Institute for Pure and Applied Mathematics (IPAM), which is supported by the National Science Foundation.}}

\begin{document}

\maketitle

\vspace{-.5cm}
{\centerline{(With an Appendix by Narutaka Ozawa)}}

\begin{abstract} 
 We show that there exist factorizable quantum channels in each dimension $\ge 11$ which do not admit a factorization through any finite dimensional von Neumann algebra, and do require ancillas of type II$_1$, thus witnessing new infinite-dimensional phenomena in quantum information theory.
We show that the set of $n \times n$ matrices of correlations arising as second-order moments of projections in finite dimensional von Neumann algebras with a distinguished trace is non-closed, for all $n \ge 5$, and we use this to
give a simplified proof of the recent result of Dykema, Paulsen and Prakash that the set of 
synchronous quantum correlations $C_q^s(5,2)$ is non-closed. Using a trick originating in work of Regev, Slofstra and Vidick, we  further show that the set of  correlation matrices  arising from second-order  moments of unitaries in finite dimensional von Neumann algebras with a distinguished trace is non-closed in each dimension $\ge 11$, from which we derive the first  result above.
\end{abstract}

\section{Introduction}

\noindent  C. Anantharaman-Delaroche introduced in \cite{A-D:factorizable}  the class of \emph{factorizable} completely positive maps between von Neumann algebras equipped with a normal faithful state, while studying non-commutative analogues of classical ergodic theory results, including, e.g., G.-C.\ Rota's ``Alternierende Verfahren'' theorem.

It was shown in \cite{HaaMusat:CMP-2011} that not all unital completely positive, trace-preserving maps on $M_n(\C)$ (also referred to as unital quantum channels in dimension $n$) are factorizable when $n\geq 3$, which also led to a negative answer to the so-called \emph{asymptotic quantum Birkoff conjecture},  \cite{SmoVerWin:entagle}. The tool was the characterization established in \cite{HaaMusat:CMP-2011} 
that a unital, completely positive, trace-preserving map 
$T \colon M_n(\C) \to M_n(\C)$ is factorizable if and only if there exist a finite von Neumann algebra $N$ equipped with a (faithful, normal) tracial state $\tau$ and a unitary $u$ in $M_n(\C) \otimes N$ such that
$T(x) = (\mathrm{id_n} \otimes \tau)(u(x \otimes 1_N)u^*)$, for all $x \in M_n(\C)$. Following the terminology introduced in \cite{HaaMusat:CMP-2015}, we say in this case that $T$ admits an \emph{exact factorization through} $M_n(\C) \otimes N$, and $N$ is called the \emph{ancilla}. In all previously studied cases of factorizable maps (see  \cite{HaaMusat:CMP-2011} and  \cite{HaaMusat:CMP-2015}, and the recent paper \cite{Muller-Hermes-Perry:4-noisy} for the case $n=2$), the ancilla could be taken to be finite dimensional, and even a full matrix algebra. It was, however, first remarked in \cite{Musat:2018} that one cannot always take the ancilla to be a full matrix algebra. In this paper we show that for each $n \ge 11$, there are factorizable maps on $M_n(\C)$ that do not  admit a finite dimensional ancilla, nor an ancilla of type I, and we give concrete examples of such maps, see Example~\ref{rm:B} and Theorem~\ref{thm:infdimfactorizable}.  Therefore, one needs to employ ancillas of type II$_1$ to describe a general factorizable channel in dimension $n\geq 11$. Observe that all factorizable quantum channels do admit an exact factorization through a type II$_1$ von Neumann algebra (even a type II$_1$ factor), by the well-known fact that each finite von Neumann algebra equipped with a fixed faithful normal tracial state embeds in a trace-preserving way into a type II$_1$ factor.  

The proof of our result uses very recent developments of analysis in  quantum information theory concerning the non-closure of certain sets of correlation matrices. The first such result was due to Slofstra, who proved the failure of what is referred to as the {\em strong Tsirelson conjecture}. This was recently refined by Dykema, Paulsen and Prakash, \cite{DykPauPra:non-closure}, to show that the set of synchronous correlation matrices $C_q^s(n,k)$ is non-closed, when $n \ge 5$ and $k \ge 2$. 

We consider the set $\mathcal{D}(n)$ of $n \times n$ matrices arising from second-order moments of $n$-tuples of projections in finite von Neumann algebras with a (normal, faithful) tracial state, and the subset $\mathcal{D}_\mathrm{fin}(n)$ consisting of those matrices that arise likewise from  $n$-tuples of projections in finite-dimensional von Neumann algebras (or \Cs s). We show that the set $\mathcal{D}_\mathrm{fin}(n)$ is not closed, when $n \ge 5$. Our proof uses  a theorem of Kruglyak, Rabanovich and Samoilenko, {\cite{KRS:Sums_2002}}, also employed in \cite{DykPauPra:non-closure}, which describes which scalar multiples of the identity operator on a (finite dimensional)  Hilbert space can arise as the sum of $n$ projections. We use this to give a shorter and  more direct proof of the main result from \cite{DykPauPra:non-closure} that the set of synchronous quantum correlation matrices $C_q^s(n,2)$ is non-closed, when $n \ge 5$. 

Kirchberg, \cite{Kir:CEP-1993}, reformulated the Connes Embedding Problem in terms of the set $\mathcal{G}(n)$ of  $n \times n$ matrices of correlations arising from unitaries in finite von Neumann algebras with a (normal, faithful) tracial state.  This result was further refined by Dykema and Juschenko, \cite{DykJus:MS2011}, and in their formulation, the Connes Embedding Problem is equivalent to the statement that $\mathcal{F}(n) = \mathcal{G}(n)$, for all $n \ge 3$, where $\mathcal{F}(n)$ is the closure of the set of $n \times n$  matrices of correlations arising from unitaries in full matrix algebras. A trick originating in (as of yet unpublished) work of Regev, Slofstra and Vidick (communicated to us by W.\ Slofstra in May 2018), which we carry out in the  setting of finite von Neumann algebras in Section 3, allows us to conclude further that the set $\mathcal{F}_{\mathrm{fin}}(2n+1)$ of  matrices of correlations arising from unitaries in finite dimensional von Neumann algebras is non-closed, whenever $\mathcal{D}_\mathrm{fin}(n)$ is non-closed, i.e., for all $n \ge 5$. 

A connection between the set  $\mathcal{G}(n)$ and the set of factorizable Schur multipliers on $M_n(\C)$ was established in \cite{HaaMusat:CMP-2015}. This connection gives the final link between the established non-closure of the sets $\mathcal{F}_{\mathrm{fin}}(2n+1)$, for $n \ge 5$, and existence of factorizable Schur multipliers with no finite dimensional ancilla (or, even stronger, non type I), in each dimension $\geq 11$.

In the Appendix, written by N.\ Ozawa, it is shown that the construction by Kruglyak, Rabanovich and Samoilenko in {\cite{KRS:Sums_2002}} of an $n$-tuple of projections with sum equal to a multiple $\alpha$ of the identity can be realized in the hyperfinite II$_1$ factor $\mathcal{R}$,  for all admissible values of $\alpha$,  except, possibly, for two extremal ones. This, in turn, implies that the factorizable Schur multipliers with no  finite dimensional ancilla found in this article \emph{do} admit $\mathcal{R}$ as an ancilla (except, possibly, for the cases corresponding to the above mentioned extremal values of $\alpha$). 

Different sets of matrices of correlations arising from the generators of L. Brown's universal C$^*$-algebra were shown to be non-closed by Harris and Paulsen \cite{Harris-Paulsen:correlation}. This was used by Gao, Harris and Junge \cite{Gao-Harris-Junge:superdense} to obtain further non-closure results in a matrix-valued setting.

\section{Non-closure of sets of matrices of quantum correlations}

\noindent  For $n\geq 2$, let $\mathcal{D}(n)$ and $\mathcal{D}_{\mathrm{fin}}(n)$ be the set of $n \times n$ matrices $\big[\tau(p_jp_i)\big]_{i,j=1}^n$, where $p_1,\dots, p_n$ are projections in some arbitrary finite von Neumann algebra, respectively, in some finite dimensional von Neumann algebra, equipped with a (normal) faithful tracial state $\tau$. We show in this section that $\mathcal{D}_{\mathrm{fin}}(n)$ is non-closed, when $n \ge 5$, and we use this to give a more direct proof, avoiding graph correlation functions, of the very recent result of Dykema, Paulsen and Prakash, \cite{DykPauPra:non-closure}, that the set of synchronous quantum correlations $C_q^s(5,2)$ is non-closed. 

Standard arguments involving ultralimits, as in the proof of Proposition~\ref{thm:maina} (v) below, respectively, direct sums of finite von Neumann algebras, show that the set 
$\mathcal{D}(n)$ is compact and convex. One can likewise show that the set $\mathcal{D}_{\mathrm{fin}}(n)$ is convex.  The subset $\mathcal{D}_{\mathrm{matrix}}(n)$ of  $\mathcal{D}_{\mathrm{fin}}(n)$  consisting of $n \times n$ matrices $\big[\mathrm{tr}_k(p_jp_i)\big]_{i,j=1}^n$, where $p_1,\dots, p_n$ are projections in a matrix algebra $M_k(\C)$, for $k \ge 1$, is not convex (and not closed), for any $n \ge 1$, since each diagonal entry of such a matrix is the (normalized) trace of a projection in a matrix algebra, which is a rational number. It is not hard to see that $\mathcal{D}_{\mathrm{matrix}}(n)$ is a dense subset of  $\mathcal{D}_{\mathrm{fin}}(n)$. 

Note that the closure of $\mathcal{D}_{\mathrm{matrix}}(n)$ is equal to $\mathcal{D}(n)$, for all $n \ge 3$, if and only if the Connes Embedding Problem has an affirmative answer (see Section~\ref{sec:unitaries} for a sketch of a proof of this fact).  Let us also observe that  $\mathcal{D}_{\mathrm{fin}}(2)= \mathcal{D}(2)$ is closed.

\begin{proposition}  We have
\begin{equation} \label{eq:D(2)}
\mathcal{D}(2) = \Big\{ \begin{pmatrix}s & u \\ u & t \end{pmatrix}:  \, 0 \le s,t \le 1, \; \max\{0,s+t-1\} \le u \le \min\{s,t\}\Big\}.
\end{equation}
Moreover, each matrix in $\mathcal{D}(2)$ arises from a pair of projections in the commutative finite-dimensional \Cs{} $\C \oplus \C \oplus \C \oplus \C$ (with respect to a suitable trace). 

In particular, it follows that $\mathcal{D}_{\mathrm{fin}}(2) = \mathcal{D}(2)$ and $\mathcal{D}_{\mathrm{fin}}(2)$ is closed.
\end{proposition}

\begin{proof} Let $(M,\tau)$ be a finite von Neumann algebra equipped with a tracial state, and let $p,q \in M$ be projections. The associated matrix in $\mathcal{D}(2)$ is
$$\begin{pmatrix} \tau(p) & \tau(pq) \\ \tau(qp) & \tau(q) \end{pmatrix} = \begin{pmatrix} s & u \\ u & t \end{pmatrix},$$
where $s = \tau(p)$, $t = \tau(q)$ and $u = \tau(pq)=\tau(qp) = \tau(pqp) = \tau(qpq)$. It is clear that $0 \le s,t \le 1$. Since $0 \le pqp \le p$ and $0 \le qpq \le q$, we further conclude that $0 \le u \le \min\{s,t\}$. 

By Kaplansky's formula (see, e.g., \cite{KadRin:opII}),  we get
$\tau(p \wedge q) = \tau(p)+\tau(q)-\tau(p \vee q) \ge \tau(p)+\tau(q)-1$. As $p \wedge q = q(p \wedge q)q \le qpq$, we infer that $u = \tau(qpq) \ge \tau(p \wedge q) \ge s+t-1$.

To complete the proof, we show that each matrix in the right-hand side of \eqref{eq:D(2)}
arises from projections $p$ and $q$ in $\C \oplus \C \oplus \C \oplus \C$, with respect to the trace with weight $(\alpha_1,\alpha_2,\alpha_3,\alpha_4)$, satisfying $\alpha_j \ge 0$ and $\sum_{j=1}^4 \alpha_j=1$. 
Set $p = (1,1,0,0)$ and $q = (1,0,1,0)$. Then $pq = (1,0,0,0)$ and 
$$\begin{pmatrix} \tau(p) & \tau(pq) \\ \tau(qp) & \tau(q) \end{pmatrix} = \begin{pmatrix} \alpha_1+\alpha_2 & \alpha_1 \\ \alpha_1 & \alpha_1+\alpha_3 \end{pmatrix}.$$
We must therefore choose $\alpha_1 = u \ge 0$, $\alpha_2 = s-u \ge 0$, $\alpha_3 = t-u \ge 0$, and $\alpha_4 = 1-\alpha_1-\alpha_2-\alpha_3$. The inequality $u \ge s+t-1$ ensures that $\alpha_1+\alpha_2+\alpha_3 \le 1$, whence $\alpha_4 \ge 0$. 
\end{proof}

\noindent
We do not know if the sets $\mathcal{D}_{\mathrm{fin}}(n)$ are closed for $n=3,4$.

We proceed to prove that the sets $\mathcal{D}_\mathrm{fin}(n)$ are non-closed, for $n \ge 5$, following the idea of Dykema, Paulsen and Prakash to use Theorem~\ref{thm:KRS} below from \cite{KRS:Sums_2002}.  As in \cite{KRS:Sums_2002}, let $\Sigma_n$ be the set of all $\alpha \ge 0$ for which there exist projections $p_1,\dots, p_n$ on a Hilbert space $H$ such that $\sum_{j=1}^n p_j = \alpha \cdot I_H$. The sets $\Sigma_n$ are completely described in  \cite{KRS:Sums_2002} and have the following properties: They are symmetric, i.e., if $\alpha \in \Sigma_n$, then $n-\alpha\in \Sigma_n$. Moreover, $\Sigma_2 = \{0,1,2\}$ and $\Sigma_3 = \{0,1,\frac32,2,3\}$. The set $\Sigma_4$ is a countably infinite  subset of the rational numbers, $\Q$, with one accumulation point, namely $2$. For $n \ge 5$, $\Sigma_n$ is the union of the interval $\big[\frac12 (n-\sqrt{n^2-4n}), \frac12 (n+\sqrt{n^2-4n})\big]$ and a countably infinite discrete subset of rational numbers, containing $\{0,1,\frac{n}{n-1}, n-\frac{n}{n-1}, n-1,n\}$. Observe that $\frac{n}{n-1} < \frac12 (n-\sqrt{n^2-4n})$. Set  
\begin{equation} \label{eq:Pi}
\Pi_n = \big[{\textstyle{\frac12}}(1-\sqrt{1-4/n}), {\textstyle{\frac12}}(1+\sqrt{1-4/n})\big] \subset \Big(\textstyle{\frac{2n-1}{n(n-1)}}, 1-\textstyle{\frac{2n-1}{n(n-1)}}\Big).
\end{equation}
Then $\Pi_n$ is an interval with non-empty interior, and $n^{-1}\Sigma_n \setminus \Pi_n$ is contained in $\Q$.

\begin{theorem}[Kruglyak, Rabanovich and Samoilenko, {\cite[Theorem 6]{KRS:Sums_2002}}]  \label{thm:KRS}
Let $n \ge 2$ be an integer. Then there exist projections $p_1, \dots, p_n$ on some \emph{finite dimensional}  Hilbert space $H$ such that $\sum_{j=1}^n p_j = \alpha \cdot I_H$ if and only if $\alpha \in \Sigma_n \cap \Q$.
\end{theorem}

\noindent  Note that the ``only if'' part of the theorem above is trivial: Apply the standard trace on $B(H)$ to both sides of the equation $\sum_{j=1}^n p_j = \alpha \cdot I_H$, to obtain $\sum_{j=1}^n \mathrm{dim}(p_j)= \alpha \,  \mathrm{dim}(H)$.  This  argument also gives the following \emph{quantitative} result for rational values of $\alpha$: If $\alpha = a/b$ is irreducible, with $a,b$ positive integers, and if there exist projections $p_1, \dots, p_n$ on a Hilbert space $H$ such that $\sum_{j=1}^n p_j = \alpha \cdot I_H$, then $\mathrm{dim}(H) \ge b$. 

For each $n \ge 2$ and each $t \in [1/n,1]$, define the $n \times n$ matrix $A^{(n)}_t = \big[A^{(n)}_t(i,j)\big]_{i,j=1}^n$  by
\begin{equation} \label{eq:A}
A^{(n)}_t(i,j) = \begin{cases} \hspace{.6cm} t, 
& i=j,\\ \displaystyle{\frac{t(nt-1)}{n-1}}, & i \ne j. \end{cases}
\end{equation}

\begin{proposition} \label{prop:main}
Let $n \ge 2$ be an integer, let $1/n \le t \le 1$, and let $\alpha = nt$. 
\begin{enumerate}
\item  Let $\cA$ be a unital \Cs{} with a faithful tracial state $\tau$ and let $p_1, \dots, p_n$ be projections in $\cA$ satisfying $\tau(p_jp_i) = A^{(n)}_t(i,j)$, for all $1\leq i, j\leq n$. Then $\sum_{j=1}^n p_j = \alpha \cdot 1_\cA$. Furthermore, if $t$ is irrational, then $\cA$ is necessarily infinite dimensional. 
Respectively, if $\alpha = a/b$ is rational with $a,b$ positive integers and $a/b$ is irreducible, then $\cA$ has no representation on a Hilbert space of dimension less than $b$.
\item Conversely, let $\cA$ be a unital \Cs{} with a tracial state $\tau$, and let $p_1,  \dots, p_n$ be projections in $\cA$ satisfying $\sum_{j=1}^n p_j = \alpha \cdot 1_\cA$. Then there exist projections $\widetilde{p}_1, \dots, \widetilde{p}_n$ in some matrix algebra $M_m(\cA)$ over $\cA$ satisfying 
$$\sum_{j=1}^n \widetilde{p}_j = \alpha \cdot 1_{M_m(\cA)}, \qquad 
\widetilde{\tau}(\widetilde{p}_j\widetilde{p}_i) = A^{(n)}_t(i,j), \quad1\leq i, j\leq  n,$$
where  $\widetilde{\tau}$ is the normalized trace on $M_m(\cA)$ induced by $\tau$.
\end{enumerate}
\end{proposition}

\begin{proof} (i). Set $q = (nt)^{-1} \sum_{j=1}^n p_j$. To show that $q=1_\cA$, it suffices to check that $1 = \tau(q) = \tau(q^2)$, since this will entail that $\tau((1-q)^*(1-q)) = \tau((1-q)^2) = 0$. These are straightforward calculations: Indeed,  $\tau(q) = (nt)^{-1} \sum_{j=1}^n \tau(p_j) = 1$, and 
$$
\tau(q^2) = (nt)^{-2}\Big(\sum_{j=1}^n \tau(p_j) + \sum_{i \ne j} \tau(p_ip_j)\Big) 
= (nt)^{-2} \left (nt + n(n-1) \cdot  \frac{t(nt-1)}{n-1}  \right ) = 1.
$$
If $t$ is irrational, then it follows  from (the easy part of) Theorem~\ref{thm:KRS} that $\cA$ does not have any finite dimensional representation on a Hilbert space, so $\cA$ must be infinite dimensional. Respectively, if $\alpha=a/b$ is rational with $a, b$ positive integers and $a/b$ irreducible, then $\cA$ has no representation on a Hilbert space of dimension less than $b$, by the (explicit) comment after the statement of Theorem~\ref{thm:KRS}.

(ii). We follow the same strategy as in the proof of \cite[Theorem 4.2]{DykPauPra:non-closure}. Let $m = n!$ and let $S_n$ denote the collection of all permutations of the set $\{1,2,\dots, n\}$. For $1\leq j\leq n$, set
$$\widetilde{p}_j = \bigoplus_{\sigma \in S_n} p_{\sigma(j)} \in \bigoplus_{\sigma \in S_n} \cA \subset M_m(\cA),$$
where the inclusion above is given by identifying the diagonal of $M_m(\cA)$ with the direct sum $\bigoplus_{\sigma \in S_n} \cA$ (for some enumeration of $S_n$).
It is easy to check that $\sum_{j=1}^n \widetilde{p}_j = \alpha \cdot 1_{M_m(\cA)}$. Moreover, for $1\leq i\leq n$,
$$\widetilde{\tau}(\widetilde{p}_i) = \frac{1}{n!} \sum_{\sigma \in S_n} \tau(p_{\sigma(i)}) = \frac{1}{n} \sum_{k=1}^n \tau(p_k) \colon \!\!= t_0,$$
and for $1\leq i \ne j\leq n$,
$$\widetilde{\tau}(\widetilde{p}_i\widetilde{p}_j) = \frac{1}{n!} \sum_{\sigma \in S_n} \tau(p_{\sigma(i)}p_{\sigma(j)}) = \frac{1}{n(n-1)} \sum_{k \ne \ell} \tau(p_k p_\ell) \colon \!\!= s_0.$$
Since  $\sum_{j=1}^n p_j = \alpha \cdot 1_\cA$, we deduce that $t_0 = \alpha/n = t$, and that
$$n^2t^2 = \alpha nt = \sum_{k=1}^n \tau(\alpha \cdot p_k) = \sum_{k=1}^n \tau\big(p_k \sum_{\ell=1}^n p_\ell\big) = \sum_{k, \ell=1}^n \tau(p_kp_\ell) = n t + n(n-1)s_0,$$
from which we conclude that $s_0 =(n-1)^{-1}t(nt-1)$, as desired.
\end{proof}

\noindent
Combining Theorem~\ref{thm:KRS} and Proposition~\ref{prop:main}, we  obtain the following:

\begin{proposition} \label{thm:maina} For $n \ge 2$ and  $1/n \le t \le 1$ we have:
\begin{enumerate}
\item $A_t^{(n)} \in \mathcal{D}(n)$ if and only if $A_t^{(n)} \in \overline{\mathcal{D}_{\mathrm{fin}}(n)}$ if and only if $t \in n^{-1} \Sigma_n$,
\item $A_t^{(n)} \in \mathcal{D}_{\mathrm{fin}}(n)$ if and only if $A_t^{(n)} \in \mathcal{D}_{\mathrm{matrix}}(n)$ if and only if  $t \in n^{-1} \Sigma_n \cap \Q$.
\end{enumerate}
Moreover, for $n \ge 5$,
\begin{itemize}
\item[\rm{(iii)}] $A_t^{(n)}$ belongs to $\overline{\mathcal{D}_{\mathrm{fin}}(n)} \setminus \mathcal{D}_{\mathrm{fin}}(n)$, for  $t \in n^{-1}\Sigma_n \setminus \Q = \Pi_n \setminus \Q$.
\item[\rm(iv)] If $t \in n^{-1}\Sigma_n \setminus \Q$ and $p_1, \dots, p_n$ are projections in a von Neumann algebra $(N,\tau)$ with faithful tracial state satisfying $A_t^{(n)}(i,j) = \tau(p_jp_i)$, for all $1\leq i,j\leq n$, then $N$ must  be of type II$_1$. 
\item[\rm(v)]  For each $A \in \overline{\mathcal{D}_{\mathrm{fin}}(n)}$, there exist projections $p_1, \dots, p_n$ in the ultrapower $\mathcal{R}^\omega$ of the hyperfinite type II$_1$ factor $\mathcal{R}$ such that $A(i,j) = \tau_{\mathcal{R}^\omega}(p_jp_i)$, for all $1\leq i,j\leq n$.
\end{itemize}
\end{proposition}

\begin{proof} Suppose that $n \ge 2$ and  $1/n \le t \le 1$. From the first part of Proposition~\ref{prop:main} (i) we see that $A_t^{(n)} \in \mathcal{D}(n)$ implies $t \in n^{-1} \Sigma_n$, while its second part shows that $A_t^{(n)} \in \mathcal{D}_{\mathrm{fin}}(n)$ implies $t \in n^{-1} \Sigma_n \cap \Q$. Proposition~\ref{prop:main} (ii) gives that $t \in n^{-1} \Sigma_n$ implies $A_t^{(n)} \in \mathcal{D}(n)$.

Let $t \in n^{-1} \Sigma_n \cap \Q$. Then by Theorem~\ref{thm:KRS}, there exist projections $p_1,\dots, p_n$ on a finite dimensional Hilbert space $H$ satisfying $\sum_{j=1}^n p_j = \alpha \cdot I_H$, where $\alpha = nt$. Identifying the bounded operators on $H$ with a full matrix algebra, we may assume that the projections $p_j$ belong to  $M_k(\C)$, for some $k \ge 2$. By Proposition~\ref{prop:main}, we can find projections $\widetilde{p}_1, \dots, \widetilde{p}_n$ in some larger matrix algebra $M_{m}(\C)$ with normalized trace $\mathrm{tr}$, satisfying $A^{(n)}_t(i,j) = \mathrm{tr}(\widetilde{p}_i\widetilde{p}_j)$, for all $1\leq i,j\leq n$. This shows that $A^{(n)}_t$ belongs to $\mathcal{D}_{\mathrm{matrix}}(n)$. As $\mathcal{D}_{\mathrm{matrix}}(n) \subset \mathcal{D}_{\mathrm{fin}}(n)$, this completes the proof of (ii). 

To complete the proof of (i) we must show that $A_t^{(n)} \in \overline{\mathcal{D}_{\mathrm{fin}}(n)}$ when $t \in n^{-1} \Sigma_n$. This follows directly from (ii) when $t$ is rational. Suppose that $t$ is irrational. As remarked below \eqref{eq:Pi}, $n^{-1}\Sigma_n \setminus \Pi_n \subset \Q$, for all $n \ge 2$. Hence $t \in \Pi_n \setminus \Q$. We can therefore find a sequence $\{t_k\}_{k=1}^\infty$  of rational numbers in the interval $\Pi_n$ converging to $t$. Then $A^{(n)}_{t_k} \to A^{(n)}_t$, as $k\rightarrow \infty$, and $A^{(n)}_{t_k}$ belongs to $\mathcal{D}_{\mathrm{fin}}(n)$, for each $k\geq 1$, by (i). This shows that $A^{(n)}_t$ belongs to the closure of $\mathcal{D}_{\mathrm{fin}}(n)$.

(iii) follows from (i) and (ii). 

(iv). The finite von Neumann algebra $N$ can have no representation on a finite dimensional Hilbert space, by the second part of Proposition~\ref{prop:main} (i), whence $N$ must be of type II$_1$.  

(v). Suppose that $A$  is the limit of a sequence $\{A_k\}_{k=1}^\infty$ of matrices in $\mathcal{D}_{\mathrm{fin}}(n)$. Since each finite dimensional \Cs{} with a distinguished trace can be embedded in a trace preserving way into the hyperfinite type II$_1$ factor $\mathcal{R}$, we can find projections $p^{(k)}_1, \dots, p^{(k)}_n$ in $\mathcal{R}$ satisfying $\tau_\mathcal{R}(p^{(k)}_jp^{(k)}_i) = A_k(i,j)$, for all $1\leq i,j\leq n$. Let $p_i$ be the image in $\mathcal{R^\omega}$ of the sequence $\{p^{(k)}_i\}_{k=1}^\infty\in \ell^\infty(\mathcal{R})$. Then 
$$\tau_{\mathcal{R}^\omega}(p_jp_i) = \lim_\omega \tau_\mathcal{R}(p^{(k)}_jp^{(k)}_i) = \lim_\omega A_k(i,j) = A(i,j),$$
for all $1\leq i,j\leq n$, as wanted.
\end{proof}

\noindent By Proposition~\ref{thm:maina} (iii) and the fact that $\Pi_n$ is an interval with non-empty interior, when $n \ge 5$, cf.\ \eqref{eq:Pi}, we obtain the following theorem:

\begin{theorem} \label{thm:mainb}
The set $\mathcal{D}_{\mathrm{fin}}(n)$ is non-compact, when $n \ge 5$.
\end{theorem}

\begin{remark} \label{rem:A}
It is shown in Proposition~\ref{thm:maina} above that one can realize the $n \times n$  matrix $A_t^{(n)}$ using projections in $\mathcal{R}^\omega$, for all $n \ge 2$ and all $t \in n^{-1}\Sigma_n$. If  moreover $t$ is rational, then $A_t^{(n)}$ can be realized using projections in some matrix algebra, by Proposition~\ref{prop:main} (ii).

Using Proposition~\ref{prop:main}, this also shows that for each $n \ge 2$ and for each $\alpha$ in $\Sigma_n$, one can find an $n$-tuple of projections summing up to $\alpha \cdot 1_N$ in \emph{some} type II$_1$ factor $N$, e.g., $N= \mathcal{R}^\omega$. One can also reach this conclusion directly from Theorem~\ref{thm:KRS}, using  an ultraproduct argument as in the  proof of Proposition~\ref{thm:maina} (v).

In Theorem~\ref{thm:hyprealization} in the Appendix by N.\ Ozawa it is shown that for each $n \ge 5$ and each $\alpha \in \big({\textstyle{\frac12}}(n-\sqrt{n^2-4n}), {\textstyle{\frac12}}(n+\sqrt{n^2-4n})\big)$, one can find an $n$-tuple of projections in $\mathcal{R}$ summing up to $\alpha \cdot 1_{\mathcal{R}}$. It follows that the matrix $A^{(n)}_t$ can be realized using projections in $\mathcal{R}$ for all $t$ in  $n^{-1}\Sigma_n$, except, possibly, for the endpoints of the interval $\big[{\textstyle{\frac12}}(1-\sqrt{1-4/n}), {\textstyle{\frac12}}(1+\sqrt{1-4/n})\big]$.
\end{remark}

\noindent
Recall, e.g., from \cite[Section 2]{DykPauPra:non-closure}, that for $n, k\geq 2$, the set $C_{{qc}}(n,k)$ consists of $nk \times nk$ quantum correlation matrices $\big[(p(i,j | v,w)\big]_{i,j,v,w}$ with entries given by
$$p(i,j | v,w) = \big\langle P_{v,i}Q_{w,j} \psi, \psi\big\rangle, \quad 1\leq i,j\leq k, \; \;  1\leq v,w\leq n,$$
where, for each $v$ and $w$, $\{P_{v,i}\}_{i=1}^k$ and $\{Q_{w,j}\}_{j=1}^k$  are projection-valued measures on some Hilbert space $H$, satisfying $P_{v,i}Q_{w,j} =Q_{w,j}P_{v,i}$, for all $i,j,v,w$, and where $\psi$ is a unit vector in $H$. 
Let $C_{{q}}(n,k)$ be the same set of quantum correlation matrices, but with the additional assumption that $H=H_A \otimes H_B$, for some {\em finite dimensional} Hilbert spaces $H_A$ and $H_B$, and $P_{v,i}$ belongs to $B(H_A) \otimes I_{H_B}$, while $Q_{w,j}$ belongs to $I_{H_A} \otimes B(H_B)$. The closure of the set $C_{{q}}(n,k)$ is denoted by $C_{{qa}}(n,k)$. We have the following inclusions: $C_q(n,k) \subseteq C_{qa}(n,k) \subseteq C_{qc}(n,k)$. 
Furthemore, the sets of \emph{synchronous} correlation matrices, denoted by $C^s_{{q}}(n,k)$,  $C^s_{{qa}}(n,k)$, and  $C^s_{{qc}}(n,k)$, respectively, consist of those quantum correlation matrices $\big[(p(i,j | v,w)\big]_{i,j,v,w}$ in $C_{{q}}(n,k)$,  $C_{{qs}}(n,k)$, and  $C_{{qc}}(n,k)$, respectively, where $p(i,j | v,v) = 0$, whenever $i \ne j$. 

We use Theorem~\ref{thm:mainb} and Proposition~\ref{prop:PSSTW} below from \cite{PSSTW:JFA-2016} to give a shorter proof of  \cite[Theorem 4.2]{DykPauPra:non-closure} of Dykema, Paulsen and Prakash, which, again, was refining Slofstra's result (in the general, non-synchronous case) from \cite{Slofstra:non-closed}. For this, using the notation from \cite{PSSTW:JFA-2016}, for $n,k \ge 2$ let $\mathcal{D}_q^s(n,k)$ be the set of matrices $\big[\tau(e_{v,i}e_{w,j})\big]_{i,j,v,w}$, where $e_{v,i}$ are projections in a finite dimensional von Neumann algebra with a (faithful) tracial state $\tau$, satisfying $\sum_{i=1}^k e_{v,i} = 1$, for all $1\leq v\leq n$. Note that $\mathcal{D}_q^s(n,k)$  is a subset of $\mathcal{D}(nk)$.

\begin{proposition}[Paulsen, Severini, Stahlke, Todorov and Winter, \cite{PSSTW:JFA-2016}] \label{prop:PSSTW}
For all $n,k \ge 2$, one has  $C_q^s(n,k) =\mathcal{D}_q^s(n,k)$.
\end{proposition}

\begin{theorem}[Dykema, Paulsen and Prakash, \cite{DykPauPra:non-closure}] \label{thm:Slofstra}
The set $C_q^s(n,2)$ of synchronous quantum correlation matrices  is non-compact, for all $n \ge 5$.
\end{theorem}

\begin{proof} Use Theorem~\ref{thm:mainb} to find a matrix $A$ in $\mathcal{D}(n) \setminus \mathcal{D}_{\mathrm{fin}}(n)$ and a sequence $\{A_k\}_{k=1}^\infty$ of matrices in $ \mathcal{D}_{\mathrm{fin}}(n)$ converging to $A$. Let $p_1, \dots, p_n$ be projections in $\mathcal{R}^\omega$ such that $\tau_{\mathcal{R}^\omega}(p_jp_i) = A(i,j)$, for all $i,j$, cf.\ Proposition~\ref{thm:maina} (v), and let further $p^{(k)}_1, \dots, p^{(k)}_n$ be projections in some matrix algebra $M_{m_k}(\C)$ with normalized trace $\mathrm{tr}_{m_k}$ such that $\mathrm{tr}_{m_k}(p^{(k)}_jp^{(k)}_i) = A_k(i,j)$, $1\leq i, j\leq n$.

Set $e_{0,v} = p_v$ and $e_{1,v} = 1-p_v$ in $\mathcal{R}^\omega$, and set $e^{(k)}_{0,v} = p_v^{(k)}$ and $e^{(k)}_{1,v} = 1- p_v^{(k)}$ in $M_{m_k}(\C)$, for all $k \ge 1$ and $1\leq v \leq n$. It then follows that
$$\lim_{k \to \infty} \big[\mathrm{tr}_{m_k}(e^{(k)}_{v,i}e^{(k)}_{w,j})\big]_{i, j, v, w} = \big[\tau_{\mathcal{R}^\omega}(e_{v,i}e_{w,j})\big]_{i, j, v, w},$$
and each of the matrices
$\big[\mathrm{tr}_{m_k}(e^{(k)}_{v,i}e^{(k)}_{w,j})\big]_{i, j, v, w}$ belongs to  $\mathcal{D}_q^s(n,2) = C_q^s(n,2)$, cf.\ Proposition~\ref{prop:PSSTW}.  However, the matrix $\big[\tau(e_{v,i}e_{w,j})\big]_{i, j, v, w}$  itself does not  belong to $C_q^s(n,2)$, since the matrix $A = \big[\tau(e_{v,0}e_{w,0})\big]_{v,w}$ does not belong to $\mathcal{D}_{\mathrm{fin}}(n)$.
\end{proof}

\noindent
Note that the set $C_q(n,2)$ of (non-synchronous) quantum correlation matrices contains  $C_q^s(n,2)$ as a relatively closed subset, which shows that $C_q(n,2)$ also is non-closed, when $n \ge 5$. 

We end this section with a remark on the matrices $A_t^{(n)}$ defined in \eqref{eq:A}.

\begin{remark}  \label{rem:B}
For each integer $n \ge 2$ and for all $s,t \in [0,1]$, consider the $n \times n$ matrix $A^{(n)}_{t,s}$, whose diagonal entries all are equal to $t$ and whose off-diagonal entries are all equal to $s$. Note that $A^{(n)}_{t} = A^{(n)}_{t,s}$, with $s := t(nt-1)/(n-1)$, when $t \in [1/n,1]$. The purpose of this remark is to describe the set of ``admissible pairs'' $(t,s)$, for which $A^{(n)}_{t,s}$ belongs to $\mathcal{D}(n)$, and to show that $s = t(nt-1)/(n-1)$ is the smallest number for which $(t,s)$ is such an admissible pair, for each fixed $t \in n^{-1}\Sigma_n$, 

Note first that for each fixed $t \in [0,1]$, the set $I^{(n)}_t$ of those $s \in [0,1]$, for which $(t,s)$ is admissible, is a closed interval. This follows by convexity and compactness of the set $\mathcal{D}(n)$.

Fix $t \in [0,1]$, take a projection $p$ of trace $t$ in some finite von Neumann algebra, and let  $p_1= \cdots = p_n = p$. The matrix $A^{(n)}_{t,t}$ is then equal to  $\big[\tau(p_jp_i)\big]_{i,j=1}^n$, and therefore belongs to $\mathcal{D}(n)$. Moreover, since $\tau(pq) \le \tau(p)$, whenever $p,q$ are projections in a von Neumann algebra with tracial state $\tau$, we conclude that  $t=\max I^{(n)}_t$.

For $0 \le t \le 1/n$ we can find pairwise orthogonal projections $p_1,\dots, p_n$ in any type II$_1$ factor $M$ with $\tau_M(p_j)=t$, for all $j$. The corresponding $n \times n$ matrix $\big[\tau(p_jp_i)\big]_{i,j=1}^n$ is equal to $A^{(n)}_{t,0}$. Hence $0$ belongs to $I^{(n)}_t$. This shows that $I^{(n)}_t = [0,t]$, when $0 \le t \le 1/n$.

Let $s,t \in [0,1]$, and suppose that $(t,s)$ is an admissible pair. Let $p_1, \dots, p_n$ be projections in a finite von Neumann algebra with tracial state $\tau$ such that $\big[\tau(p_jp_i)\big]_{i,j=1}^n = A^{(n)}_{t,s}$. Put $q_j = 1-p_j$, $1 \le j \le n$. Then $\big[\tau(q_jq_i)\big]_{i,j=1}^n = A^{(n)}_{1-t,1-2t+s}$, which shows that $(1-t,1-2t+s)$ is an admissible pair. The map given by $(t,s) \longmapsto (1-t,1-2t+s)$ is involutive, and  therefore it maps the set of admissible pairs in $[0,1]^2$ onto itself. In particular, this involution maps 
$\{t\} \times I_t^{(n)}$ onto $\{1-t\} \times I_{1-t}^{(n)}$. Combining this fact with the result of the previous paragraph, we obtain that $I_t^{(n)} = [2t-1,t]$, when $1-1/n\le t \le 1$.

Consider finally the case where $t \in [1/n,1-1/n]$. Suppose that $s \in I_t$, and let $p_1, \dots, p_n$ be pro\-jec\-tions in some finite von Neumann algebra with tracial state $\tau$ such that $\big[\tau(p_jp_i)\big]_{i,j=1}^n = A^{(n)}_{t,s}$. Then
$$nt + n(n-1)s =  \tau\Big( (\sum_{i=1}^n p_i)^2\Big) \ge \Big(\tau(\sum_{i=1}^n p_i)\Big)^2 = n^2 t^2,$$
which implies that $s \ge  t(nt-1)/(n-1)$. In other words, $I_t^{(n)} \subseteq [ t(nt-1)/(n-1),t]$. 
It was noted in Remark~\ref{rem:A} that $A^{(n)}_t= A^{(n)}_{t,t(nt-1)/(n-1)}$  belongs to $\mathcal{D}(n)$ if and only if $t$ belongs to $n^{-1}  \Sigma_n$. For those values of $t$, we therefore obtain that  $s = t(nt-1)/(n-1)$ is the smallest number for which $(t,s)$ is an admissible pair, whence $I^{(n)}_t = [ t(nt-1)/(n-1),t]$.

It remains a curious open problem to detemine the interval $I^{(n)}_t$, for $t$ belonging to the (non-empty) set $[1/n, 1-1/n] \setminus n^{-1} \, \Sigma_n$. In this case, necessarily, $\min I^{(n)}_t > t(nt-1)/(n-1)$.
\end{remark}

\section{Correlation matrices of unitary elements} \label{sec:unitaries}

\noindent Recall that a correlation matrix is a positive definite matrix whose diagonal entries  are equal to $1$. For each integer $n \ge 2$, let $\mathcal{G}(n)$, $\mathcal{F}_{\mathrm{matrix}}(n)$, and $\mathcal{F}_{\mathrm{fin}}(n)$ be the set of $n \times n$ correlation matrices $\big[\tau(u_j^*u_i)\big]_{i,j=1}^n$, where $u_1,u_2, \dots, u_n$ are unitaries in some finite von Neumann algebra equipped with a faithful tracial state $\tau$, respectively, in some full matrix algebra with its canonical tracial state, respectively, in some  finite dimensional $C^*$-algebra with a faithful tracial state. Then the convex hull $\mathrm{conv}(\mathcal{F}_{\mathrm{matrix}}(n))$  of $\mathcal{F}_{\mathrm{matrix}}(n)$ is equal to $\mathcal{F}_{\mathrm{fin}}(n)$, and the two sets $\mathcal{F}_{\mathrm{matrix}}(n)$ and  $\mathcal{F}_{\mathrm{fin}}(n)$ have the same closure, which is denoted by $\mathcal{F}(n)$. The sets $\mathcal{G}(n)$ and $\mathcal{F}(n)$ are compact and convex, see \cite[Proposition 1.4]{DykJus:MS2011}.

As shown by Kirchberg,  \cite{Kir:CEP-1993} (cf.\ Dykema--Juschenko, \cite{DykJus:MS2011}), 
the Connes Embedding Problem has an affirmative answer if and only if $\mathcal{G}(n) = \mathcal{F}(n)$,  for all $n \ge 3$. We can use this to show that $\mathcal{D}_{\mathrm{matrix}}(n)$ is dense in $\mathcal{D}(n)$, for all $n \ge 3$, if and only if the Connes Embedding Problem has an affirmative answer. Indeed, the proof of the ``if'' part follows the same strategy as  the proof of the ``if'' part of Kirchberg's result,  where a given finite subset of $\mathcal{R}^\omega$ is approximated in trace-norm with a finite subset of a matrix subalgebra of $\mathcal{R}^\omega$. To see the ``only if'' part, assume that $\mathcal{D}_{\mathrm{matrix}}(n)$ is dense in $\mathcal{D}(n)$, for all $n \ge 3$. We show that this implies $\mathcal{G}(n) = \mathcal{F}(n)$,  for all $n \ge 3$. Take an $n$-tuple $u_1, \dots, u_n$ of unitaries in some tracial von Neumann algebra $(N,\tau)$. Fix an integer $m \ge 1$. Each $u_j$ can be approximated in norm within $2\pi/m$ by unitaries $v_1, \dots, v_n$ in $N$ of the form $v_j = \sum_{k=1}^m \omega^k p_{k,j}$, where $\omega = \exp(2\pi i/m)$ and $p_{1,j}, \dots, p_{m,j}$ are pairwise orthogonal projections in $N$ summing up to $1$, for each $j$. Approximate the second-order moments of the collections of projections $\{p_{k,j}\}_{k,j}$ by second-order moments of projections $\{q_{k,j}\}_{k,j}$ in some matrix algebra $(M_r(\C), \mathrm{tr}_r)$. Then $\mathrm{tr}_r(q_{k,j}q_{\ell,j})$ are small when $k \ne \ell$, and $\mathrm{tr}_r\big(\sum_{k=1}^m q_{k,j}\big)$ is close to $1$, for all $j$. A standard lifting argument allows us to replace the projections $\{q_{k,j}\}_{k,j}$ with new projections, close to the old ones in trace-norm, satisfying $\sum_{j=1}^m q_{k,j} = 1$, for all $j$. Set $w_j = \sum_{k=1}^m \omega^k q_{k,j}$. The second-order moments of the unitaries $w_1, \dots, w_n$ are then close to those of $u_1, \dots, u_n$.

\begin{remark} In the case where $n=2$, the set  $\mathcal{F}_{\mathrm{matrix}}(2)$ is closed and convex, and $\mathcal{F}_{\mathrm{matrix}}(2) = \mathcal{F}(2) = \mathcal{G}(2)$, which further is equal to the set of $2 \times 2$ matrices of the form
$$\begin{pmatrix} 1 & \bar{z} \\ z & 1 \end{pmatrix},$$
for $z \in \C$ with $|z| \le 1$. To see that each of these $2 \times 2$ matrices belongs to $\mathcal{F}_{\mathrm{matrix}}(2)$, take $z \in \C$ with $|z| \le 1$ and find $\lambda_1$ and $\lambda_2$ on the complex unit circle $\T$, with $z = (\lambda_1+\lambda_2)/2$. The correlation matrix arising from the unitary $2 \times 2$ matrices $u_1=1$ and $u_2 = \mathrm{diag}(\lambda_1,\lambda_2)$ is then as desired.
\end{remark}

We show in this section that the set $\mathcal{F}_{\mathrm{fin}}(n)$ is not closed (hence, not compact), for 
all $n \ge 11$.  This result originates in a remark made by T.\ Vidick during his talk at one of the workshops in the Quantitative Linear Algebra program at IPAM, Spring 2018, that led to subsequent discussions with W. Slofstra, who, in particular, communicated to us a version of the following result (to appear, in an approximate case, in a forthcoming paper by O.\ Regev, W.\ Slofstra and T.\ Vidick):

\begin{proposition} \label{prop:1}
Let $M$ be a finite von Neumann algebra with a faithful tracial state $\tau_M$, and let $p_1,\dots, p_n$ be projections in $M$. Further, let $u_0,u_1, \dots, u_n$, $u_{n+1}, \dots, u_{2n}$ be the unitaries in $M$ given by 
$$u_0 = 1, \qquad u_j = 2p_j - 1, \; \; (1 \le j \le n), \qquad u_j = (u_{j-n} + i \cdot 1)/{\sqrt{2}}, \; \; (n+1 \le j \le 2n).$$
Let $N$ be another finite von Neumann algebra with a faithful tracial state $\tau_N$. Then there exist $2n+1$ unitaries $v_0,v_1, \dots, v_{2n}$  in $N$ satisfying
\begin{equation} \label{*}
\tau_N(v_j^*v_i) = \tau_M(u_j^*u_i), \qquad 0 \leq i,j\leq 2n,
\end{equation}
if and only if there exist $n$ projections $q_1, \dots, q_n$ in $N$ satisfying
\begin{equation} \label{**}
 \tau_N(q_jq_i) = \tau_M(p_jp_i),\qquad 1\leq  i,j\leq n.
\end{equation}
\end{proposition}

\begin{proof} Assume that $q_1, \dots, q_n$ are projections in $N$  satisfying \eqref{**}. Equip the vector spaces $\mathrm{span} \{p_1,p_2, \dots, p_n\}$ and 
$\mathrm{span}\{q_1,q_2, \dots, q_n\}$ with the Euclidean structure arising from the traces $\tau_M$ and $\tau_N$, respectively.
Using \eqref{**}, we see that the map $p_j \longmapsto q_j$, $1 \le j \le n$, extends to a well-defined linear isometry $\varphi$ from $\mathrm{span}\{p_1, \dots, p_m\} $ to $\mathrm{span}\{q_1, \dots, q_m\}$. 
Set $v_0 = 1$, $v_j = 2q_j-1$, for $1 \le j \le n$, and $v_j = (v_{j-n}+i \cdot 1)/\sqrt{2}$, for $n+1 \le j \le n$, and use the isometric property of $\varphi$ to check that \eqref{*} holds. 

Conversely, assume that we are given unitaries $v_0,v_1, \dots, v_{2n}$  in $N$ satisfying \eqref{*}. Upon replacing  $v_j$ by $v_0^*v_j$, for all $0\leq j\leq 2n$, we may assume that $v_0=1$. As above, equip the vector spaces $\mathrm{span} \{u_0,u_1, \dots, u_{2n}\}$ and
$\mathrm{span}\{v_0,v_1, \dots, v_{2n}\}$ with the Euclidean structure arising from the traces $\tau_M$ and $\tau_N$, respectively. Then, by \eqref{*}, we have a well-defined linear isometry $\psi \colon \mathrm{span} \{u_0,u_1, \dots, u_{2n}\}  \to \mathrm{span}\{v_0,v_1, \dots, v_{2n}\}$, mapping $u_j$ to $v_j$, for $0 \le j \le 2n$. 
In particular, for $1 \le j \le n$,
$$\big\|v_{j+n} - (v_j + i \cdot v_0)/\sqrt{2}\big\|_2 =  \big\|u_{j+n} - (u_j + i \cdot u_0)/\sqrt{2}\big\|_2 = 0,$$
so $v_{j+n} =  (v_j + i \cdot 1)/\sqrt{2}$.

Note that if $u$ and $(u+i \cdot 1)/\sqrt{2}$ are unitaries in some unital $C^*$-algebra, then $u$ is necessarily a symmetry.  Indeed, if $\lambda$ is a complex number such that $ |\lambda| = |(\lambda + i)/\sqrt{2}| = 1$, then $\lambda \in \R$. Hence, if $u$ is as stated, then its spectrum is contained in $\R$, which entails that it is a symmetry.

We conclude that $v_1, \dots, v_n$ are symmetries. For $1\leq j\leq  n$, set $q_j = (v_j+1)/2$. Then $q_j$ is a projection and $v_j = 2q_j-1$. Use the isometric property of $\psi$ to check \eqref{**}.
\end{proof}

\begin{corollary} The set $\mathcal{F}_{\mathrm{matrix}}(m)$ is not compact and not convex, whenever $m \ge 3$.
\end{corollary}

\begin{proof} Let  $0 < \alpha < 1$ be irrational. Equip $M := \C \oplus \C$ with the trace $\tau$ given by $\tau(x,y) = \alpha x + (1-\alpha)y$, for $x, y\in \mathbb{C}$.

 Consider first the case where $m=3$.  Let $n=1$ and let $p=p_1 = (1,0) \in M$. Then $\tau(p) = \alpha$ is irrational.  Let $u_0,u_1,u_2$ be the unitaries in $M$ arising from this projection as in the proposition above (with $n=1$). The  matrix $\big[\tau(u_j^*u_i)\big]_{i,j=0}^{2}$ belongs to $\mathrm{conv}(\mathcal{F}_{\mathrm{matrix}}(3))$, and hence to $\mathcal{F}(3)$. However, by Proposition~\ref{prop:1}, it does not belong to $\mathcal{F}_{\mathrm{matrix}}(3)$ itself,  because no full matrix algebra contains a projection of irrational trace, and therefore contains no projection $q=q_1$ satisfying \eqref{**} (with $n=1$).

Assume now that $m > 3$, let $u_0,u_1,u_2$ be as above, and let unitaries $u_3,u_4, \dots, u_{m-1} \in M$ be arbitrary.  If $v_0,v_1, \dots, v_{m-1}$ are unitaries in some tracial von Neumann algebra $(N,\tau_N)$ satisfying \eqref{*}, then $v_0,v_1,v_2$ satisfy \eqref{*} with respect to the set $\{u_0,u_1,u_2\}$, so  $v_0,v_1,v_2$, and hence $v_0,v_1, \dots, v_{m-1}$ cannot be found in a full matrix algebra.
These arguments also yield the non-convexity of the set $\mathcal{F}_{\mathrm{matrix}}(m)$ in all cases.
\end{proof}

\begin{example}
For $m=3$, the correlation matrix $B=\big[\tau(u_j^*u_i)\big]_{i,j=0}^{2}$ from the proof above, with unitaries given by
$$u_0 = (1,1), \qquad u_1 = (1,-1), \qquad u_2 = \left(\frac{1+i}{\sqrt{2}}, \frac{-1+i}{\sqrt{2}}\right),$$
has the following explicit form in terms of the parameter $\alpha \in (0,1)$:
$$B = \begin{pmatrix} 1 & \gamma& \frac{\gamma - i}{\sqrt{2}} \\
\gamma & 1 & \frac{1-\gamma i}{\sqrt{2}}\\
\frac{\gamma + i}{\sqrt{2}}& \frac{1+\gamma i}{\sqrt{2}}& 1
\end{pmatrix},$$
where $\gamma = 2\alpha-1 \in (-1,1)$.  Note that the matrix $B$ belongs to $\mathcal{F}_{\mathrm{fin}}(3)$, for all $\gamma \in (-1,1)$, while it does not belong to $\mathcal{F}_{\mathrm{matrix}}(3)$, whenever $\gamma$ is irrational. 
\end{example}

\begin{example} \label{rm:B}
For each $n \ge 2$ and each $t \in [1/n,1]$, consider the self-adjoint $(1+2n) \times (1+2n)$ complex matrix
$$B_t^{(n)} = \begin{pmatrix}1 & X^* & Y^* \\ X & D_1 & C^* \\ Y & C & D_2 \end{pmatrix},
$$
where $X, Y$ and $C, D_1, D_2$ are the $n \times 1$, respectively, $n \times n$ complex matrices given by 
$$X = s \begin{pmatrix}1 \\ 1 \\ \vdots \\ 1 \end{pmatrix}, \qquad Y = \frac{s+i}{\sqrt{2}} \begin{pmatrix} 1\\ 1 \\ \vdots \\ 1 \\\end{pmatrix}, \qquad  C = \frac{1+is}{\sqrt{2}} \, I_n + \frac{4(r-t)+1+is}{\sqrt{2}} \, E,
$$
$$D_1 = I_n + (4(r-t)+1) \, E, \qquad D_2 = I_n + (2(r-t)+1) \, E, \quad$$
where $s = 2t-1$ and  $r = (n-1)^{-1}t(nt-1)$, and 
$$E = \begin{pmatrix} 0 & 1 & \cdots & 1 \\ 1 & 0 &  \cdots & 1 \\ \vdots & \vdots & \ddots  & \vdots \\ 1 & 1  & \cdots & 0
\end{pmatrix}\in M_n(\mathbb{C}).$$
With this definition, Proposition~\ref{prop:1} yields that the $(2n+1) \times (2n+1)$ matrix $B^{(n)}_t$ is the correlation matrix of an $(2n+1)$-tuple of unitaries in some finite von Neumann algebra $M$ if and only if the $n \times n$ matrix $A_t^{(n)}$ is the correlation matrix of an $n$-tuple of projections in the same von Neumann algebra $M$. Indeed, if $p_1,\dots, p_n$ are projections in $M$ such that $\big[\tau_M(p_jp_i)\big]_{i,j=1}^n = A_t^{(n)}$, and if $u_0,u_1, \dots, u_{2n}$ are the $2n+1$ unitaries in $M$ constructed from these projections as in Proposition~\ref{prop:1}, then $B_t^{(n)} = \big[\tau_M(u_j^*u_i)\big]_{i,j=0}^{2n}$. Conversely, if $v_0,v_1, \dots, v_{2n}$ are unitaries in $M$ such that $B_t^{(n)} = \big[\tau_M(v_j^*v_i)\big]_{i,j=0}^{2n}$, then there are projections $q_1, \dots, q_n \in M$ such that $A_t^{(n)} = \big[\tau_M(q_jq_i)\big]_{i,j=1}^n$.

In particular, $B^{(n)}_t$ belongs to $\mathcal{G}(2n+1)$, respectively, to $\mathcal{F}_{\mathrm{fin}}(2n+1)$, if and only if $A^{(n)}_t$ belongs to $\mathcal{D}(n)$, respectively, to $\mathcal{D}_{\mathrm{fin}}(n)$.
\end{example}

\begin{theorem} \label{thm:F_n-non-compact}
Let $n \ge 5$ and let $t \in \Pi_n$. 
\begin{enumerate}
\item Then $B_t^{(n)}$ belongs to $\mathcal{F}_{\mathrm{fin}}(2n+1)$, if $t$ is rational, and $B_t^{(n)}$ belongs to $\mathcal{F}(2n+1) \setminus \mathcal{F}_{\mathrm{fin}}(2n+1)$, if $t$ is irrational. 
\item If $t$ is irrational and if $v_0,v_1, \dots, v_{2n}$ are unitaries in some finite von Neumann algebra $N$ with a faithful  tracial state $\tau_N$ such that $\tau_N(v_j^*v_i) = B_t^{(n)}(i,j)$, for $0 \le i,j \le 2n$, then $N$ is necessarily of type II$_1$.
\item If $t \in \big({\textstyle{\frac12}}(1-\sqrt{1-4/n}), {\textstyle{\frac12}}(1+\sqrt{1-4/n})\big)$, then there are unitaries $v_0,v_1, \dots, v_{2n}$ in the hyperfinite II$_1$ factor $\mathcal{R}$ such that $\tau_N(v_j^*v_i) = B_t^{(n)}(i,j)$, for $0 \le i,j \le 2n$.
\item The convex sets  $\mathcal{F}_{\mathrm{fin}}(k)$ are non-compact, for all $k \ge 11$.
\end{enumerate}
\end{theorem}

\begin{proof}  (i). The map $t \mapsto B_t^{(n)}$, $t \in \Pi_n$, is clearly continuous. It follows from Proposition~\ref{prop:main} and Example~\ref{rm:B} that  $B_t^{(n)} \in \mathcal{G}(2n+1)$, whenever $A_t^{(n)} \in \mathcal{D}(n)$, and in particular whenever $t \in \Pi_n$.  Moreover, if $t \in \Pi_n$, then $B_t^{(n)} \in \mathcal{F}_{\mathrm{fin}}(2n+1)$, when $t$ is rational, and $B_t^{(n)} \notin  \mathcal{F}_{\mathrm{fin}}(2n+1)$, when $t$ is irrational. This implies that $B_t^{(n)} \in \mathcal{F}(2n+1)$, for all $t \in \Pi_n$, and that $\mathcal{F}_{\mathrm{fin}}(2n+1)$ is non-compact, for each $n \ge 5$.

(ii). Let $v_0,v_1, \dots, v_{2n}$ be unitaries in some finite von Neumann algebra $N$ with faithful  tracial state $\tau_N$, satisfying $\tau_N(v_j^*v_i) = B_t^{(n)}(i,j)$, for $0 \le i,j \le 2n$. Then, by Example~\ref{rm:B}, there exist projections $q_1,\dots, q_n$ in $N$ such that $A_t^{(n)}(i,j) = \tau_N(q_jq_i)$. This entails that $N$ has no finite dimensional representations, by Proposition~\ref{prop:main},  so $N$ must be of type II$_1$. 

(iii). Let $t \in \big({\textstyle{\frac12}}(1-\sqrt{1-4/n}), {\textstyle{\frac12}}(1+\sqrt{1-4/n})\big)$. It follows from Theorem~\ref{thm:hyprealization} in the Appendix by Ozawa that $A_t^{(n)}$ can be realized using projections in $\mathcal{R}$, cf.\ Remark~\ref{rem:A}. The claim now follows from Example~\ref{rm:B}.

(iv). We show that if $\mathcal{F}_{\mathrm{fin}}(k)$ is non-compact, for some positive integer $k$, then so is $\mathcal{F}_{\mathrm{fin}}(k+1)$. To this end, define a map $\rho \colon \mathcal{F}(k) \to \mathcal{F}(k+1)$ by
$$\rho\Big(\big[\tau_M(u^*_ju_i)\big]_{i,j=1}^k\Big) = \big[\tau_M(u^*_ju_i)\big]_{i,j=1}^{k+1},$$
whenever $u_1, \dots, u_k$ are unitaries in some von Neumann algebra $M$ with a faithful tracial state $\tau$, and where $u_{k+1}$ is chosen to be equal to $u_1$. Then the last row and the last column of $\big[\tau_M(u^*_ju_i)\big]_{i,j=1}^{k+1}$ is equal to the first row and the first column of the matrix $\big[\tau_M(u^*_ju_i)\big]_{i,j=1}^k$, which shows that $\rho$ is well-defined and continuous. Moreover, $\rho^{-1}(\mathcal{F}_{\mathrm{fin}}(k+1)) = \mathcal{F}_{\mathrm{fin}}(k)$. Hence, if $\mathcal{F}_{ \mathrm{fin}}(k+1)$ is compact, and thus closed, then $\mathcal{F}_{\mathrm{fin}}(k)$ is closed, and thus compact. 
\end{proof}

\section{Factorizable maps that require infinite dimensional ancilla}

\noindent We prove here our claimed result about existence of factorizable quantum channels in all dimensions $\geq 11$, requiring infinite dimensional ancilla. We first recall necessary prerequisites.

Let $\mathrm{UCPT}(n)$ denote the convex and compact set of all unital completely positive trace preserving linear maps $T \colon M_n(\C) \to M_n(\C)$, $n\geq 2$. Maps in $\mathrm{UCPT}(n)$ are also called unital quantum channels in dimension $n$. Anantharaman-Delaroche defined in \cite{A-D:factorizable} a channel to be \emph{factorizable} if it admits a factorization (in a suitable way) through a finite von Neumann algebra with a faithful tracial state. This notion was studied extensively in \cite{HaaMusat:CMP-2011}, and the following characterization (which we will take to be our definition of factorizable maps) was established therein (cf., \cite[Theorem 2.2]{HaaMusat:CMP-2011}):
A unital quantum channel $T$ in dimension $n$ is factorizable if and only if there exist a finite von Neumann algebra $N$, equipped with a normal faithful tracial state $\tau_N$,  and a unitary $u \in M_n(\C) \otimes N$ such that
\begin{equation} \label{HMsf}
T(x) = (\mathrm{id}_n \otimes \tau_N)(u(x \otimes 1_N)u^*), \qquad x \in M_n(\C).
\end{equation}
The von Neumann algebra $N$ above is also called the \emph{ancilla}, and, following Definition 3.1 in  \cite{HaaMusat:CMP-2015}, we say that $T$ has an {\em exact factorization} through $M_n(\C) \otimes N$. The set of all factorizable unital channels in dimension $n$ is denoted by $\mathcal{FM}(n)$. This set is convex and compact, as shown by standard arguments.

Further, let $\mathcal{FM}_{\mathrm{matrix}}(n)$ and $\mathcal{FM}_{\mathrm{fin}}(n)$ denote the set of factorizable maps in UCPT$(n)$ that exactly factor through a full matrix algebra, respectively, through a finite dimensional $C^*$-algebra, equipped with a faithful tracial state. 
It was shown in \cite[Theorem 3.7]{HaaMusat:CMP-2015} that a positive answer to the Connes Embedding Problem is equivalent to $\mathcal{FM}_{\mathrm{matrix}}(n)$ being dense in $\mathcal{FM}(n)$, for all $n \ge 3$. Moreover, if a unital quantum channel $T\colon M_n(\mathbb{C})\rightarrow M_n(\mathbb{C})$ belongs to the closure of $\mathcal{FM}_{\mathrm{matrix}}(n)$, then $T$ admits an exact factorization through an ultrapower $\mathcal{R}^\omega$ of the hyperfinite type II$_1$ factor $\mathcal {R}$.

It is shown in the upcoming manuscript \cite{Musat:2018} that $\mathcal{FM}_{\mathrm{fin}}(n) = \mathrm{conv}(\mathcal{FM}_{\mathrm{matrix}}(n))$, and that $\mathcal{FM}_{\mathrm{matrix}}(n)$ is non-compact and non-convex, when $n \ge 3$. 

Let $\mathcal{S}(n)$ be the set of all \emph{Schur multipliers} $T_B \colon M_n(\C) \to M_n(\C)$, where $B \in M_n(\C)$, i.e., $T_B(x)$ is the Schur product of $B$ and $x$, for $x \in M_n(\C)$. Let $\mathcal{FMS}(n) = \mathcal{FM}(n) \cap \mathcal{S}(n)$ be the set of all factorizable Schur multipliers, and write $\mathcal{FMS}_{\mathrm{fin}}(n) = \mathcal{FM}_{\mathrm{fin}}(n) \cap \mathcal{S}(n)$.

For the theorem below, recall the definition \eqref{eq:Pi} of the set $\Pi_n$, and the $(2n+1) \times (2n+1)$ matrix $B_t^{(n)}$ constructed in Example~\ref{rm:B}.

\begin{theorem} \label{thm:infdimfactorizable}
The set $\mathcal{FM}_{\mathrm{fin}}(k)$ is not compact, for all $k \ge 11$. 
Moreover,  for each $n \ge 5$ and  each irrational number $t \in \Pi_n$, the Schur multiplier $T_B$, where $B = B_t^{(n)}$, is a factorizable map belonging to the closure of  $\mathcal{FM}_{\mathrm{fin}}(2n+1)$, but not to $\mathcal{FM}_{\mathrm{fin}}(2n+1)$, and it requires an ancilla of type II$_1$. This ancilla can be taken to be the hyperfinite II$_1$ factor $\mathcal{R}$, when ${\textstyle{\frac12}}(1-\sqrt{1-4/n}) < t < {\textstyle{\frac12}}(1+\sqrt{1-4/n})$.
\end{theorem}

\begin{proof} It was shown in \cite[Proposition 2.8]{HaaMusat:CMP-2011} that if $B \in M_k(\C)$ is a correlation matrix, then the associated Schur multiplier $T_B$ admits an exact factorization through $M_k(\C) \otimes N$, where $N$ is a finite von Neumann algebra with normal faithful tracial state $\tau_N$, if and only if $B = \big[\tau_N(u_j^*u_i)\big]_{i,j=1}^k$, for some unitaries $u_1,\dots, u_k \in N$.  In particular,
$T_B $ belongs to $\mathcal{FMS}(k)$, $\mathcal{FMS}_{\mathrm{fin}}(k)$, and the closure of $\mathcal{FMS}_{\mathrm{fin}}(k)$, respectively, if and only if $B$ belongs to $\mathcal{G}(k)$, $\mathcal{F}_{\mathrm{fin}}(k)$, and $\mathcal{F}(k)$, respectively.

Since the map from $\mathcal{G}(k)$ to $\mathcal{FM}(k)$ given by $B \mapsto T_B$  is continuous, it follows from Theorem~\ref{thm:F_n-non-compact} that $\mathcal{FMS}_{\mathrm{fin}}(k)$ is non-closed in $\mathcal{FM}(k)$. As the set $\mathcal{S}(k)$ is closed, we conclude that $\mathcal{FM}_{\mathrm{fin}}(k)$ is non-compact.

To prove the second part of the theorem, let $n \ge 5$, let $t \in \Pi_n$ be irrational, and let $B=B_t^{(n)}$. Then $B$ belongs to $\mathcal{F}(2n+1)\setminus \mathcal{F}_{\mathrm{fin}}(2n+1)$ by Theorem~\ref{thm:F_n-non-compact}. By the argument in the first paragraph, we conclude that  $T_B$ belongs to the closure of $\mathcal{FMS}_{\mathrm{fin}}(2n+1)$, but not to $\mathcal{FMS}_{\mathrm{fin}}(2n+1)$. Moreover, if $T_B$ admits an exact factorization through $M_n(\C) \otimes N$, with ancilla $(N,\tau_N)$ as in the first paragraph, then $B$ is the matrix of correlations of unitaries $u_0,u_1, \dots, u_{2n}$ in $N$, which by Theorem~\ref{thm:F_n-non-compact} implies that $N$ must be of type II$_1$. 

The remaining part of the theorem follows from Theorem~\ref{thm:F_n-non-compact}~(iii) and the argument in the first paragraph.
\end{proof}

 We also obtain the following quantitative version of the theorem above, in the case where 
$t \in \Pi_n$ is rational and $n \ge 5$: Let $B = B_t^{(n)}$.  If $nt = a/b$, with $a,b$ positive integers and $a/b$  irreducible, then the Schur channel $T_B$ admits an exact factorization through a finite dimensional von Neumann algebra $M= M_n(\C) \otimes N$, where the von Neumann algebra $N$ can have no representation on a Hilbert space of dimension smaller than $b$. This follows as in the proof of the theorem above and by appealing to Proposition~\ref{prop:main}~(i). 

We conclude that for every fixed integer $n \ge 11$, there is a sequence of factorizable unital  quantum channels in dimension $n$, each admitting finite-dimensional ancillas, but where the size of any such ancillas must tend to infinity.

\newpage

\appendix

{\centerline{{\Large{\bf{Appendix}}}}

\section*{Realizing the Kruglyak--Rabanovich--Samoilenko projections in the hyperfinite II$_1$ factor}

\bigskip

\centerline{{\Large{by Narutaka Ozawa}}}

\bigskip \noindent Kruglyak, Rabanovich, and Samoilenko proved in  {\cite[Theorem 6]{KRS:Sums_2002}}, cf.\ Theorem~\ref{thm:KRS}, that 
for any $n\geq5$ and 
any $\alpha \in [\frac{1}{2}(n-\sqrt{n^2-4n}),\frac{1}{2}(n+\sqrt{n^2-4n})]$, 
there are orthogonal projections $p_1,\ldots,p_n$ such that 
$\sum_i p_i=\alpha$.
In this appendix, we observe that the Kruglyak--Rabanovich--Samoilenko 
construction shows that 
these projections are realized 
in the hyperfinite $\mathrm{II}_1$ factor $\cR$, 
possibly except for the extremities. 

\begin{customthm}{A.1}
\label{thm:hyprealization}
For any $n \ge 5$ and any 
$\alpha \in (\frac{1}{2}(n-\sqrt{n^2-4n}),\frac{1}{2}(n+\sqrt{n^2-4n}))$, 
there are projections $p_1,\ldots,p_n\in \cR$ which satisfy 
$\sum_i p_i=\alpha$.
\end{customthm}

Let $(M,\tau)$ be a finite von Neumann algebra. 
By a \emph{matricial approximation} (or matricial microstates) 
of a $d$-tuple $(a_1,\ldots,a_d)$ 
in $M^\mathrm{sa}$, we mean a sequence 
$(x_1(n),\ldots,x_d(n)) \in (M_{k(n)}(\C)^\mathrm{sa})^d$ 
such that 
$\lim_n \tr(p(x_1(n),\ldots,x_d(n))) = \tau(p(a_1,\ldots,a_d))$ 
for every polynomial $p$ in $d$ non-commuting variables. 
A matricial approximation of a generating  $d$-tuple $(a_1,\ldots,a_d)$ 
of $M$ gives rise to an embedding 
of $(M,\tau)$ into the tracial ultraproduct of $(M_{k(n)}(\C),\tr_{k(n)})$.
Recall that $M$ satisfies the Connes Embedding Conjecture,
i.e., $(M,\tau)\hookrightarrow(\cR^\omega,\tau^\omega)$,
if and only if every (or some) generating $d$-tuple $(a_i)_{i=1}^d$ 
in $M^\mathrm{sa}$ admits a matricial approximation. 
We will give a sufficient condition for hyperfiniteness of $M$ in terms 
of a matricial approximation. 
For $x = (x_{ij}) \in M_k(\C)$, we define its \emph{propagation} 
to be $\max\{ |i-j| : x_{ij}\neq0\}$. 

\begin{customlm}{A.2}
\label{lem:unifbddprop}
Let $(M,\tau)$ be a finite von Neumann algebra generated by 
$a_1,\ldots,a_d \in M^\mathrm{sa}$. 
Assume that $(a_1,\ldots, a_d)$ admits a matricial approximation 
$(x_1(n),\ldots,x_d(n))$ with uniformly bounded propagations.
Then $M$ is hyperfinite. 
\end{customlm}

\begin{proof}
Let $k$ be a positive integer and consider the shift unitary matrix
$z_k\in M_k({\mathbb C})$ given by $(z_k)_{i,j}=\delta_{i+1,j}$
(modulo $k$).
It normalizes the diagonal maximal abelian subalgebra
$D_k\subset M_k({\mathbb C})$.
Observe that any $y \in M_k(\C)$ that has propagation at most $l$
can be written as $y=\sum_{m=-l}^l f_m z_k^m$ for some $f_m\in D_k$
with $\|f_m\|\le\|y\|$.

Now, let a matricial approximation $(x_1(n),\ldots,x_d(n))$
be given as in the statement.
We denote by $M_\omega$ the tracial ultraproduct of
$(M_{k(n)}({\mathbb C}),\mathrm{tr}_{k(n)})$ and by
$D_\omega$ the subalgebra arising from the diagonal maximal
abelian subalgebras $D_{k(n)}$.
 From the above discussion, one sees that the element
$x_i\in M_\omega$ that corresponds to $(x_i(n))_n$ belongs to
the von Neumann subalgebra generated by $D_\omega$ and $z$,
where $z$ is the unitary element corresponding to $(z_{k(n)})_n$.
The von Neumann subalgebra generated by $x_1,\ldots,x_d$ is
isomorphic to $M$ and the von Neumann subalgebra generated by
$D_\omega$ and $z$ is hyperfinite (as it is isomorphic to
$D_\omega\rtimes {\mathbb Z}$, assuming $k(n)\to\infty$).
\end{proof}

\begin{proof}[Proof of Theorem~\ref{thm:hyprealization}]
Firstly, note that every separable
finite von Neumann algebra $(N,\tau)$ with a faithful normal tracial state  is embeddable in a trace-preserving way into a separable $\mathrm{II}_1$ factor $M$, which can be taken to be the hyperfinite II$_1$ factor $\cR$ if $M$ is hyperfinite. Indeed, as observed by U.\ Haagerup, we may take $M$ to be $(\bigotimes_{n=1}^\infty N) \rtimes S_\infty$, where the infinite tensor product is with respect to the standard representation of $N$ on $L^2(N,\tau)$, and where $S_\infty$ is the (locally finite) group of permutations on the natural numbers with finite support. It therefore suffices to find the projections $p_1, \dots, p_n$ in any hyperfinite finite von Neumann algebra $N$. 

For each $\alpha\in{\mathbb Q}\cap[3/2,2]$, the projections $P_1(\alpha),\ldots,P_5(\alpha)$ 
in $M_{k(\alpha)}(\C)$ that satisfy $\sum_i P_i(\alpha)=\alpha$ are constructed in 
\cite[Theorem 6]{KRS:Sums_2002}
as $R_i$. 
The proof of Theorem~6 (and Lemma~7) in \cite{KRS:Sums_2002} reveals that 
the projections $R_i$ are obtained by \emph{sewing} (see \cite[Definition~1]{KRS:Sums_2002}) 
the projections $P^{(k)}_i \in M_{k_i+2}(\C)$, $k_i\in\{1,2,3\}$. 
Since $P^{(k)}_i$'s have propagation at most $4$, 
the projections $R_i$ have propagation at most $8$, regardless of $\alpha$. 

Let $\alpha\in(3/2,2)$ be given and take a rational sequence $(\alpha_n)_n$ 
which converges to $\alpha$. 
Then after passing to a convergent subsequence, 
$(P_1(\alpha_n),\ldots,P_5(\alpha_n))$ is a matricial approximation 
of $(P_1,\ldots, P_5)$ in the tracial ultraproduct $M_\omega$ 
of $(M_{k(n)}(\C),\tr_{k(n)})$ and $(P_1,\ldots, P_5)$ 
satisfies $\sum_i P_i=\alpha$. 
By Lemma~\ref{lem:unifbddprop}, the projections 
$P_1,\ldots, P_5$ generate a hyperfinite von Neumann subalgebra. 
This proves Theorem~\ref{thm:hyprealization} for $n=5$ and $\alpha\in[3/2,2]$. 
By \cite[Lemma 5]{KRS:Sums_2002}, this implies Theorem~\ref{thm:hyprealization} for 
every $n\geq5$ and $\alpha\in[2,n-2]$.

Finally, note that all values in 
$ (\frac{1}{2}(n-\sqrt{n^2-4n}),\frac{1}{2}(n+\sqrt{n^2-4n}))$ are 
obtained by iterating the numerical mappings $\Phi^+$ and $\Phi^-$ (see \cite[Section 1.2]{KRS:Sums_2002}) 
starting at $\alpha\in[2,n-2]$ (see \cite[Lemma~6]{KRS:Sums_2002}).
Thus it suffices to show the functors $S$ and $T$ constructed 
in Section 1.2 in \cite{KRS:Sums_2002} preserve hyperfiniteness. 
This is clear for the linear reflection $T$. 
For the reader's convenience, we replicate here the construction of 
the hyperbolic reflection $S$, adapted to our setting. 
Let $P_1,\ldots,P_n\in N$ be projections such that $\sum_{i=1}^n P_i=\alpha$. 
We will construct projections $Q_1,\ldots,Q_n$ such that $\sum_{i=1}^n Q_i=\frac{\alpha}{\alpha-1}$. 
Put 
\[
V_i:=(\alpha^2-\alpha)^{-1/2} \, P_i \left[\begin{matrix} -P_1 & \cdots & \alpha-P_i &\cdots&-P_n\end{matrix}\right] \in M_{1,n}(N).
\]
Then, 
$V_iV_i^*=(\alpha^2-\alpha)^{-1}P_i(\alpha^2-2\alpha P_i + \sum_{k=1}^n P_k)P_i=P_i$ and 
$V_i$ is a partial isometry. 
Hence $Q_i:=V_i^*V_i \in M_n(N)$ is a projection. A calculation shows 
$$\sum_k Q_k=(\alpha^2-\alpha)^{-1}(\alpha^2 \, \mathrm{diag}(P_1,\ldots,P_n) - \alpha \, [P_iP_j]_{i.j} )\in M_n(N).$$
Note that $Q:=\mathrm{diag}(P_1,\ldots,P_n) - \alpha^{-1} [P_iP_j]_{i.j} $ is a projection and 
one has $\sum_k Q_k=\frac{\alpha}{\alpha-1}Q$. 
Thus, viewing $Q_k$ as projections in $QM_n(N)Q$, we are done. 
When $N$ is hyperfinite, so is the amplification $QM_n(N)Q$.
\end{proof}

{\small{
\bibliographystyle{amsplain}
\providecommand{\bysame}{\leavevmode\hbox to3em{\hrulefill}\thinspace}
\providecommand{\MR}{\relax\ifhmode\unskip\space\fi MR }
\providecommand{\MRhref}[2]{%
  \href{http://www.ams.org/mathscinet-getitem?mr=#1}{#2}
}
\providecommand{\href}[2]{#2}

}}

\vspace{1cm}

\noindent
\begin{tabular}{ll}
Magdalena Musat & Mikael R\o rdam \\
Department of Mathematical Sciences & Department of Mathematical Sciences\\
University of Copenhagen & University of Copenhagen\\ 
Universitetsparken 5, DK-2100, Copenhagen \O & Universitetsparken 5, DK-2100, Copenhagen \O \\
Denmark & Denmark\\
musat@math.ku.dk & rordam@math.ku.dk
\end{tabular}

\vspace{1.5cm}

\noindent 
\begin{tabular}{l}
Narutaka Ozawa \\
RIMS, Kyoto University\\
Sakyo-ku, Kyoto 606-8502\\
Japan\\
narutaka@kurims.kyoto-u.ac.jp\\
\end{tabular}

\end{document}